\documentclass[12pt,a4paper]{amsart}
\usepackage{enumerate}
\usepackage{graphicx}
\usepackage{yhmath}
\usepackage{amsmath,verbatim}
\usepackage{amssymb}
\usepackage{pgf,tikz}
\usetikzlibrary{arrows}
\usepackage[utf8]{inputenc}

\usepackage{hyperref}
\usepackage{cleveref}
\usepackage{physics}

\theoremstyle{plain}
\newtheorem{thm}{Theorem}[section]
\newtheorem{mthm}{Theorem}
\newtheorem{lemma}[thm]{Lemma}

\newtheorem{prop}[thm]{Proposition}
\newtoks\prt

\theoremstyle{definition}

\newtheorem*{ex*}{Example}

\def\eqn#1$$#2$${\begin{equation}\label#1#2\end{equation}}

\numberwithin{equation}{section}

\def\diam{\operatorname{diam}}

\def\L{\mathcal{L}}

\def\H{\mathcal{H}}
\def\phi{\varphi}

\def\epsilon{\varepsilon}

\def\er{\mathbb R}

\newcommand{\lip}{{\mathrm {lip}}}

\newcommand{\diff}{{\mathrm{d}}}
\newcommand{\DIFF}{{\mathrm{D}}}

\newcommand{\per}{{\mathrm {Per}}}

\newcommand{\mres}{\mathbin{\vrule height 1.6ex depth 0pt width 0.13ex\vrule height 0.13ex depth 0pt width 1.3ex}}

\newcommand{\dist}{{\mathsf{d}}}
\newcommand{\mass}{{\mathsf{m}}}

\newcommand{\XX}{{\mathsf{X}}}
\newcommand{\YY}{{\mathsf{Y}}}

\newcommand{\defeq}{\mathrel{\mathop:}=}

\newcommand{\HH}{\mathcal{H}}
\newcommand{\LL}{\mathcal{L}}
\newcommand{\RR}{\mathbb{R}}
\newcommand{\NN}{\mathbb{N}}

\newcommand{\Lp}{{\mathrm {L}}}
\newcommand{\Lploc}{\mathrm {L}_\mathrm{loc}}

\newcommand{\LIP}{{\mathrm {LIP}}}

\newcommand{\BV}{{\mathrm {BV}}}
\newcommand{\W}{\mathrm {W}^{1,1}}
\newcommand{\BVloc}{{\mathrm {BV_{loc}}}}
\newcommand{\Wloc}{\mathrm {W_{loc}^{1,1}}}

\newcommand{\LIPloc}{{\mathrm {LIP_{loc}}}}

\newcommand{\PI}{{\mathrm {PI}}}

\def\Xint#1{\mathchoice
	{\XXint\displaystyle\textstyle{#1}}%
	{\XXint\textstyle\scriptstyle{#1}}%
	{\XXint\scriptstyle\scriptscriptstyle{#1}}%
	{\XXint\scriptscriptstyle\scriptscriptstyle{#1}}%
	\!\int}
\def\XXint#1#2#3{{\setbox0=\hbox{$#1{#2#3}{\int}$}
		\vcenter{\hbox{$#2#3$}}\kern-.5\wd0}}

\def\dashint{\Xint-}

\MakeRobust{\ref}

\makeatletter
\newcommand{\labeltext}[2]{%
	\@bsphack
	\def\@currentlabel{#1}{\label{#2}}%
	\@esphack
}
\makeatother

\begin{document}
	
	\title[BV and Sobolev homeomorphisms]{BV and Sobolev homeomorphisms between \\metric measure spaces and the plane}
	
	\author[C. Brena]{Camillo Brena}
	\address{C.~Brena: Scuola Normale Superiore, Piazza dei Cavalieri 7, 56126 Pisa, Italy} 
	\email{\tt camillo.brena@sns.it}
	
	\author[D. Campbell]{Daniel Campbell}
	\address{D.~Campbell: Department of Mathematics, University of Hradec Kr\' alov\' e, Rokitansk\'eho 62, 500 03 Hradec Kr\'alov\'e, Czech Republic} 
	\address{Faculty of Economics, University of South Bohemia, Studentsk\' a 13, Cesk\' e Budejovice, Czech Republic}
	\email{\tt daniel.campbell@uhk.cz}

	\thanks{The second author was supported by the grant GACR 20-19018Y}

	
	\begin{abstract}
	We show that given a homeomorphism $f:G\rightarrow\Omega$ where $G$ is a open subset of $\RR^2$ and $\Omega$ is a open subset of a $2$-Ahlfors regular metric measure space supporting a weak $(1,1)$-Poincaré inequality, it holds $f\in\BVloc(G,\Omega)$ if and only $f^{-1}\in\BVloc(\Omega,G)$. Further if $f$ satisfies the Luzin N and N$^{-1}$ conditions then $f\in \Wloc(G,\Omega)$ if and only if $f^{-1}\in \Wloc(\Omega,G)$.

	\end{abstract}
	
	\maketitle
	
	\section{Introduction}
	In 2007, \cite{HeKoOnHomeoBV}, Hencl, Koskela and Onninen proved that a planar homeomorphism is in BV if and only if its inverse is BV with the variation of the inverse bounded by a constant multiple of the variation of the map. This result was enhanced by D'Onofrio and Schiattarella \cite{donoschi} in 2013 by  including an equality between variations of the map and variations of its inverse.
	
	The study of BV maps on metric measure spaces dates back to the seminal paper \cite{MIRANDA2003} (see also \cite{amb00,amb01}). A natural extension of the result in \cite{HeKoOnHomeoBV} is to ask whether the result still holds when we replace the target space or the domain with a metric measure space. In general we cannot expect to achieve the D'Onofrio, Schiattarella type estimate because it depends on a standard choice of metric and measure on the spaces involved. To see this it suffices to consider the identity map from the Euclidean space with the Lebesgue measure to the Euclidean space with the Lebesgue measure multiplied by a constant density not equalling 1. In particular, the form of Theorem \ref{mainimproved} below can not be improved.
	
	Further motivation for our study arises from an attempt to extend novel techniques for approximating homeomorphisms with diffeomorphisms (see \cite{IKO,HP,DPP}) in the plane to the $\er^3$ context. A core component of each of these methods is the use of `grids', i.e.\ the separation of the plane (either in the image or preimage or both) into rectangular parts. Typically, this is done in such a way that the restriction of the map (or its inverse) to the boundaries of the rectangles belongs to the corresponding Sobolev or BV space. Hence we want to know whether the restriction of a Sobolev (or BV) homeomorphism from $\er^3$ to $\er^3$ onto (almost any) hyperplane has BV inverse. We show that if the image of the plane is a sufficiently regular metric measure space, then the claim is true. Our main result is as follows:

	\begin{mthm}\label{main}
		Let $(\XX,\dist,\mass)$ be a $2$-Ahlfors regular metric measure space supporting a weak $(1,1)$-Poincaré inequality. Let moreover $G\subseteq\RR^2$ open, $\Omega\subseteq\XX$ open and $f:G\rightarrow\Omega$ a homeomorphism. Then $f\in\BVloc(G,\Omega)$ if and only if $f^{-1}\in \BVloc(\Omega,G)$.
	\end{mthm}
Further we prove, as a corollary, the following
	\begin{mthm}\label{mainplus}
	Let $(\XX,\dist,\mass)$ be a $2$-Ahlfors regular metric measure space supporting a weak $(1,1)$-Poincaré inequality. Let moreover $G\subseteq\RR^2$ open, $\Omega\subseteq\XX$ open and $f:G\rightarrow\Omega$ a homeomorphism such that 
\begin{align}\label{Luz}\tag{N}
&\mass(f(N))=0&&\text{if $N\subseteq G$ is a Borel set with }\LL^2(N)=0,
\\
\label{Luz1}\tag{N$^{-1}$}
&\LL^2(f^{-1}(N))=0&&\text{if $N\subseteq \Omega$ is a Borel set with }\mass(N)=0.
\end{align}
Then $f\in\Wloc(G,\Omega)$ if and only if $f^{-1}\in \Wloc(\Omega,G)$.
\end{mthm}
It is common to refer to \eqref{Luz} as the Luzin N condition satisfied by $f$, whereas  \eqref{Luz1} states the Luzin N condition satisfied by $f^{-1}$, or, in other words, the Luzin N$^{-1}$ condition satisfied by $f$.
	

	\subsection{Overview of the approach of the proof}
	To prove the fact that a homeomorphism between open planar sets has locally bounded variation if and only its inverse has locally bounded variation, in \cite{HeKoOnHomeoBV} \cite{donoschi} different strategies were used. In particular, it is not clear how to adapt the approximation strategy used in the former to our context. On the other hand, the ``slicing'' strategy employed by the latter can be adapted, which allows us to prove that a BV homeomorphism from the plane into a metric measure space has BV inverse. Of course, when the original and target space are both $\er^2$ this is enough to conclude, but this is not our case. We therefore have to find a way to ``reverse'' the technique.
	
	The main obstacle to the approach is that we have to recover a quantitative $\HH^1$ bound for the boundaries of particular sets of finite perimeter in terms of their perimeter (the sets of finite perimeter will be the ones enclosed by Jordan curves -- in this way we have a definite interior and exterior to the curve). This will be Lemma \ref{Jordanmain} whose proof, in the context of (regular enough) metric measure spaces faces the difficulty of the lack of a powerful smooth differential structure. Indeed, when we try to approximate the boundary of a set of finite perimeter, the first attempt is to regularize the set and then look at level sets of functions in an approximating sequence. In $\RR^2$ we can choose level sets that have a nice ``piecewise curve'' structure, using a combination of Sard's theorem and the implicit function theorem, but this is definitely not easy in the context of metric measure spaces. We show that such a procedure can be carried out also in our context, where the approximating sequence is obtained through discrete convolutions and a weak ``piecewise curve'' structure is obtained by a purely topological argument.
	
	Of course, we cannot hope to generalize the result of \cite{HeKoOnHomeoBV} unless we add assumptions to the metric measure space. In particular, the class of functions of bounded variation into a metric space depends on the metric, while the class of functions of bounded variation on a metric measure space depends both on the metric and the measure. For this reason, conditions assuring a sufficient compatibility between metric and measure are clearly needed. We add examples showing that without the $2$-Ahlfors regularity the result is false, and we remark here that the weak $(1,1)$-Poincaré inequality is needed only in one direction, to require a deeper compatibility between metric and measure.   
	
	\section{Preliminaries}
	\subsection{General material}
	Besides the Euclidean setting (for which we adopt the usual notation $\dist_e$ for the distance and $\LL^n$ for the $n$ dimensional Lebesgue measure) we will work with metric measure spaces. An metric measure space (m.m.s.\ for short) is a triplet $(\XX,\dist,\mass)$ where $\XX$ is a set, $\dist$ is a complete and separable distance on $\XX$ and $\mass$ is a Borel measure that is finite on balls. 
	We exclude the case in which an m.m.s.\ is a single point and we adopt the convention that an m.m.s.\ has full support, that is to say that for any $x\in\XX$, $r>0$, we have $\mass(B_r(x))>0$. Here $B_r(x)$ denotes the open ball of radius $r$ centred at $x$. We will assume that balls have a definite centre and radius, although there may be more than one couple centre-radius describing the same subset of $\XX$. If $B=B_r(x)$ is ball, by $\alpha B$ we denote $B_{\alpha r}(x)$ whenever $\alpha>0$.
	When relevant we emphasize the dependence on the metric as follows $B_r^\dist(x)$. If $A\subseteq\XX$, by $B_r(A)$ we denote $\bigcup_{x\in A}B_r(x)$. As customary, if $A\subseteq\XX$ we write $\mathrm{dist}(x,A)\defeq\inf_{y\in A}\dist(x,y)$. If $A\subseteq\XX$ is open we write $A'\Subset A$ if $A'$ is a bounded subset of $A$ with $\dist (A',\XX\setminus A)>0$. Clearly, if the space is proper (i.e.\ bounded sets are relatively compact), $A'\Subset A$ if and only if $A'$ has compact closure contained in $A$.  Of course, some of the notions above and several ones in what follows in this introductory section make sense in the framework of metric spaces, not necessarily endowed with a measure. We implicitly extend all these notions to the more general framework.

	If $f$ is a Borel function and $A\subseteq\XX$ is open, we write $$ \dashint_A f\dd{\mass}\defeq f_{A}\defeq\frac{1}{\mass(A)}\int_A f\dd{\mass}$$
	whenever it makes sense (e.g.\ $f\in\Lp^1(A)$).
	
	Given $A\subseteq\XX$ open, we denote with $\LIP(A)$ the space of functions that are $L$-Lipschitz on $A$ for some $L\in(0,\infty)$, that means that $\abs{f(x)-f(y)}\le L\dist(x,y)$ for every $x,y\in A$.
	$\LIPloc(A)$ denotes the space of functions that are Lipschitz in a neighbourhood of $x$ for any $x\in A$. If the space is locally compact, $\LIPloc(A)$ coincides with the space of functions that are Lipschitz on  compact subsets of $A$.
	Given $f\in\LIPloc(A)$, we define on $A$ the Borel function
	$$ \lip(f)(x)\defeq\limsup_{y\rightarrow x} \frac{\abs{f(y)-f(x)}}{\dist(y,x)}$$
	where $\lip(f)(x)$ is understood to be $0$ if $x$ is an isolated point.
	
		An m.m.s.\ $(\XX,\dist,\mass)$ is said to be doubling if there exists a constant $C_D\ge 1$ such that $$ \mass(B_{2 r}(x))\le C_D\mass(B_{r}(x))\quad\text{for every }r>0\text{ and }x\in\XX.$$ 
	It is easy to show that in a doubling m.m.s.\ balls are totally bounded, hence the space is proper.
	A more stringent notion is the one of Ahlfors regularity: more precisely, given $\nu>0$, we say that $(\XX,\dist,\mass)$ is $\nu$-Ahlfors regular if there exist constants $0<C'_H\le C''_H$ such that 
	\begin{equation}\label{Ahlfors}
C'_H r^{\nu} \le\mass(B_r(x))\le C''_H r^{\nu}\quad\text{for every }0<r<\diam(\XX)\text{ and }x\in\XX. 
	\end{equation}

	\subsection{Hausdorff measures}
	Given $\nu\in(0,\infty)$, we define the $\nu$-dimensional Hausdorff measure through the Carathéodory construction (\cite[2.10]{Federer69}) as $$\HH^\nu(B)\defeq\sup_{\delta>0}\HH^\nu_\delta(B),$$ where 
	$$ 
	\HH^\nu_\delta(B)\defeq\frac{\omega_\nu}{2^\nu}\inf\left\{\sum_{i\in\NN} \diam(B_i)^\nu:\diam(B_i)<\delta,\  B\subseteq\bigcup_{i\in\NN} B_i \right\}
	$$
	where $$\omega_\nu\defeq\frac{\pi^{\nu/2}}{\Gamma(\nu/2+1)}$$ and $\Gamma$ is the Euler's gamma function. Notice that $\omega_k=\LL^k(B_1(0))$ in $\RR^k$ if $k\in\NN$, $k\ge 1$.
	
	If $(\XX,\dist,\mass)$ is a doubling m.m.s.\ the natural measure is the $1$-codimensional Hausdorff measure, defined through the Carathéodory construction (\cite[2.10]{Federer69} again) as
	$$\HH^h(B)\defeq\sup_{\delta>0}\HH^h_\delta(B),$$
	where 
	$$ 
	\HH^h_\delta(B)\defeq\inf\left\{\sum_{i\in\NN} h(B_{r_i}(x_i)):r_i<\delta,\  B\subseteq\bigcup_{i\in\NN} B_{r_i}(x_i) \right\}
	$$
	where 
	$$h(B_{r_i}(x_i))\defeq\frac{\mass(B_{r_i}(x_i))}{r_i}.$$
	Recall that \cite[2.10]{Federer69} implies that the two classes of measures defined above are Borel measures.
	
		A crucial property that we are going to exploit in the derivation of the main result is that if $(\XX,\dist,\mass)$ is $2$-Ahlfors regular, then the measures $\HH^h$ and $\HH^1$ are comparable, in the sense that there exists $C>0$ such that for every $B\subseteq\XX$ Borel 
	\begin{equation}\label{comaprehauss}
		C^{-1}\HH^h(B)\le \HH^1(B)\le C\HH^h(B).
	\end{equation}

	\subsection{Curves}
A curve $\gamma$ is a continuous (possibly non injective) map $\gamma:[0,l]\rightarrow\XX$, where $l>0$. A Jordan curve is a curve $\gamma$ such that $$x<y\text{ and }\gamma(x)=\gamma(y)\quad\text {if and only if}\quad x=0\text{ and }y=l.$$
	In other words, a Jordan curve is an embedding ${S}^1\rightarrow\XX$, up to reparametrizations. We will often denote, with a slight abuse, with $\gamma$ both the curve and the image of the curve (its support).
	Given a curve $\gamma$, we define its length as 
\begin{equation}\label{defnlen}
L(\gamma)\defeq\sup\left\{\sum_{i=1}^N \dist(\gamma_{t_i},\gamma_{t_{i-1}}):0\le t_0\le t_1\dots,t_N\le l\right\} 
\end{equation}
	and by \cite[Theorem 2.10.13]{Federer69}  the following formula holds $$L(\gamma)=\int_\XX N(\gamma|[0,l], x) \, \dd \H^1(x)
	$$
	where $N(\gamma|A, x) \defeq \H^0(\gamma^{-1}(x)\cap A)$. In particular, if $\gamma$ is either injective or a Jordan curve, then 
	\begin{equation}\label{leneH}
L(\gamma)=\HH^1(\gamma).
	\end{equation}
If $L(\gamma)<\infty$ we say that $\gamma$ is rectifiable and, if this is the case, following e.g.\ \cite[2.5.16]{Federer69} we can parametrize $\gamma$ by arc-length, $\gamma:[0,L(\gamma)]\rightarrow\XX$.

 We say that a curve $\gamma$ is absolutely continuous and we write $\gamma\in\mathrm{AC}([0,l], \XX)$ if there exists $g\in\Lp^1([0,l])$ such that $$\dist(\gamma_t,\gamma_s)\le\int_s^t g(r)\dd{r} \quad\text{for every }0\le s\le t\le l $$
and we define also $\abs{\dot{\gamma_t}}$, the metric speed of $\gamma$, as the minimal (in the $\LL^1$-a.e.\ sense) function for which the above inequality holds. Moreover it holds that (\cite[Theorem 1.1.2]{AmbrosioGigliSavare08})
$$ \abs{\dot{\gamma_t}}=\lim_{s\rightarrow t}\frac{\dist(\gamma_s,\gamma_t)}{\abs{s-t}}\quad \text{for $\LL^1$-a.e.\ }t.$$

We say that a metric space $(\XX,\dist)$ is quasi-convex if the there exists a constant $C_Q>0$ such that for every $x,y\in\XX$ there exists a curve $\gamma:[0,1]\rightarrow \XX$ with $L(\gamma)\le C_Q\dist(x,y)$ such that $\gamma_0=x$ and $\gamma_1=y$. A quasi-convex space with $C_Q=1$ is called geodesic.

\begin{lemma}[{\cite[Theorem 4.4.7]{AmbrosioTilli04}}]\label{FullCream}
Let $(\XX,\dist)$ be a complete metric space and $C\subseteq\XX$ a closed connected set such that $\HH^1(C)<\infty$. Then $C$ is compact and arcwise connected, in the sense that given $x,y\in C$, we can find an injective curve $\eta:[0,1]\rightarrow C$ with $\eta(0)=x$ and $\eta(1)=y$. 
\end{lemma}

	 \subsection{Differential structure}
	 Besides $2$-Ahlfors regularity, the other standing assumption will be that $(\XX,\dist,\mass)$ supports a weak $(1,1)$-Poincaré inequality, that is 
	 that there exist constants $C_P>0$ and $\lambda\ge1$ such that
	 \begin{equation}\label{Poincare}
	\dashint_{B_r(x)}\abs{f-f_{B_r(x)}}\dd{\mass}\le C_P r \dashint_{B_{\lambda r}(x)} g\dd{\mass}
	 \end{equation}
	 whenever $f\in\Lploc^1(\XX)$ and $g$ is an upper gradient of $f$ according to the notion introduced by Heinonen and Koskela in \cite{hkquasi} (other references on this topic are e.g.\ \cite{Bjorn-Bjorn11,Hajasz:1996aa,Shanmugalingam00}). Precisely, a Borel function $g$ is an upper gradient of a Borel function $f$ if 
	 $$ \abs{f(\gamma_1)-f(\gamma_0)}\le \int_0^1 g(\gamma_t)\abs{\dot{\gamma_t}}\dd{t}\quad\text{for every }\gamma\in\mathrm{AC}([0,1], \XX).$$
	 If $f\in\LIPloc(\XX)$, then it is clear that $\lip(f)$ is an upper gradient of $f$ and, if the space is doubling and supports a weak $(1,1)$-Poincaré inequality, it is minimal, in the sense that if $g$ is another upper gradient for $f$ on $A$, then $\lip(f)\le g\ \mass$-a.e.\ (\cite{Cheeger00}).
	 
	 It is common to denote a doubling m.m.s.\ that supports a weak $(1,1)$-Poincaré inequality $\PI$ space. This structure has several important consequences, among them, the fact that $\PI$ spaces are quasi-convex (\cite[Theorem 3.2]{amb01}). Taking into account that $\PI$ spaces are proper, it follows that there exists a geodesic metric $\tilde{\dist}$ that is bi-Lipschitz to $\dist$, in the sense that  $$\dist(x,y)\le\tilde{\dist }(x,y)\le C_Q{\dist }(x,y)\quad\text{for every }x,y\in\XX.$$
	 If $(\XX,{\dist},\mass)$ is a geodesic $\PI$ space (e.g.\ $(\XX,\tilde{\dist},\mass)$ as above), it was shown in \cite[Corollary 9.8]{hajkosk} that it supports a $(1,1)$-Poincaré inequality, i.e.\ a weak $(1,1)$-Poincaré inequality with $\lambda=1$.
	 
	\subsection{Functions of bounded variation on and into metric measure spaces}
We assume that the reader is familiar with the basics of the theory of functions of bounded variation and sets of finite perimeter in metric measure spaces developed in \cite{amb00,amb01,MIRANDA2003,ambmetric}. We recall here the main notions for the reader's convenience.

Let $(\XX,\dist,\mass)$ be an m.m.s.\ Given $A\subseteq\XX$ open and $f\in\Lploc^1(A)$, we define the total variation
\begin{equation}\notag
	\abs{\DIFF f}(A)\defeq\inf\left\{\liminf_k \int_\XX\lip(f_k)\dd{\mass} :f_k\in\LIPloc(A)\cap\Lploc^1(A),\ f_k\rightarrow f \text{ in }\Lploc^1(A)\right\}
\end{equation}
and we write $f\in\BV(A)$ if $\abs{\DIFF f}(A)<\infty$ and $f\in\BVloc(A)$ if $f\in\BV(A')$ for every $A'$ open with $A'\Subset A$. When we feel the need to specify the reference measure, we write $\BV(A;\mass)$ or $\BVloc(A;\mass)$. If $\chi_E$ is the characteristic function of a Borel subset $E\subseteq A$, we say that $E$ is a set of locally finite perimeter in $A$ provided that $\chi_E\in\BVloc(A)$ and, of course, we say that $E$ is a set of finite perimeter in $A$ if $\chi_E\in\BV(A)$.  If $f\in\BVloc(A),$ $\abs{\DIFF f}(\,\cdot\,)$ turns out to be the restriction to open subsets of $A$ of a Borel measure that we still denote with the same symbol and we call total variation. If $f=\chi_E$, we denote $\abs{\DIFF f}(\,\cdot\,)$ also with $\per(E,\,\cdot\,)$.
It holds the following coarea formula.
\begin{prop}[Coarea]\label{coarea}
	Let $(\XX,\dist,\mass)$ be an m.m.s.\ Given $A\subseteq\XX$ open and $f\in\Lploc^1(A)$, then 
	\begin{equation}\label{coareaeq}
		{\abs{\DIFF f}}(A)=\int_\RR {\per(\{f>r\},A)}\dd{r} .
	\end{equation}
	In particular, if the right hand side of \eqref{coareaeq} is finite, then $f\in\BV(A)$.
\end{prop}
We also recall that sets of finite perimeter are an algebra, more precisely, if $E$ and $F$ are sets of (locally) finite perimeter, then 
\begin{equation}\notag
	\per(E,\,\cdot\,)=\per(\XX\setminus E,\,\cdot\,)\quad\text{and}\quad\per(E\cap F,\,\cdot\,)+\per(E\cup F,\,\cdot\,)=\per(E,\,\cdot\,)+\per(F,\,\cdot\,).
\end{equation}
Notice that $\abs{\DIFF (\varphi \circ f)}\le L \abs{\DIFF f}$ whenever $f\in\BV(A)$ and $\varphi$ is $L$-Lipschitz. Also, it is easy to verify that if $(\XX,\dist,\mass)=(\RR^n,\dist_e,\LL^n)$ then the definition of total variation, and hence of function of (locally) bounded variation, coincides with the classical one.

We will also need the following that is \cite[Lemma 3.1]{lahti2016strong}. 
\begin{lemma}\label{LipEst}
	Let $\XX$ be a $(\XX,\dist ,\mass)$ be a $\PI$ space. Then there exists a constant $C>0$ such that for any $U\subseteq \XX$ open and $f \in \LIPloc(U)$ it holds that
	\begin{equation}\label{LipEsteq}
		\int_{-\infty}^{\infty}\HH^h(U\cap \partial \{f> t\}) \dd{t} \le C \int_U \lip(f)  \dd \mass.
	\end{equation}
\end{lemma}
Notice that under the additional assumption of $2$-Ahlfors regularity, \eqref{comaprehauss} implies that we can read \eqref{LipEsteq} as 
\begin{equation}\label{LipEsteq1}
	\int_{-\infty}^{\infty}\HH^1(U\cap \partial \{f> t\}) \dd{t} \le C \int_U \lip(f)  \dd \mass;
\end{equation}
also, for $\L^1$ almost every $t$ it holds that $\mass(\{f=t\}\cap U)=0$ and for such $t$ it holds
$$
\{f=t\}\cap U\subseteq (\partial \{f>t\}\cap U)\cup(\partial \{f<t\}\cap U)
$$
and therefore
\begin{equation}\label{LipEsteq2}
\int_\RR\HH^1( \{f=t\}\cap U)\dd{t}\le C\int_U\lip(f)\dd{\mass}.
\end{equation}

Let now $(\YY,\rho)$ be a locally compact and separable metric space. Given $A\subseteq\XX$ open and $f:A\rightarrow\YY$ such that for some (hence all) $\bar{y}\in\YY$ it holds $\rho(f(\,\cdot\,),\bar{y})\in\Lploc^1(A)$ we define (see \cite[Definition 2.1]{ambmetric})
$$\abs{\DIFF f}\defeq \mathcal{M} - \sup\left\{\abs{\DIFF (\varphi\circ f)}:\varphi\in\LIP(\YY) \text{ is $1$-Lipschitz}\right\}  $$
and extend obviously the definitions of the classes $\BV(A)$ and $\BVloc(A)$ to this setting. Notice that in the case $(\YY,\dist)=(\RR,\dist_e)$ this definition coincides with the one given above, and, if $(\YY,\dist)=(\RR^n,\dist_e)$ and $f=f_1,\dots,f_n$ then $f\in\BV(A)$ if and only if $f_i\in\BV(A)$ for every $i=1,\dots,n$.

Assume now $(\XX,\dist,\mass)=(\RR,\dist_e,\LL^1)$. Let $\gamma:[0,l]\rightarrow\YY$ be a curve.
Then \cite[Remark 2.2]{ambmetric} implies that (recall \eqref{defnlen}) $$L(\gamma)=\abs{\DIFF \gamma}(0,l)$$
and, in particular, if $\gamma$ is either injective or a Jordan curve, then, by \eqref{leneH}, 
\begin{equation}\label{diffeH}
	\abs{\DIFF \gamma}(0,l)=\HH^1(\gamma).
\end{equation}

If $G\subseteq\RR^2$ is open and $(\YY,\rho)$ is a locally compact metric space, we will often consider a map $f:G\rightarrow\YY$. We denote $G_x\defeq\{y\in\RR:(x,y)\in G\}$ and similarly $G^y\defeq\{x\in\RR:(x,y)\in G\}$. We then consider the restriction of $f$ to the lines $\{x\}\times G_x$ as $f_x$ i.e.\ $f_x:G_x\rightarrow\YY$ is defined as $y\mapsto f(x,y)$. Similarly, $f^y:G^y\rightarrow\YY$ is defined as $x\mapsto f(x,y)$.
We recall that \cite[Proposition 2.1]{ambmetric} implies that  \begin{equation}\label{slicingeq}
	\int_\RR \abs{\DIFF f_x}(G_x)\dd{x}\le\abs{\DIFF f}(G),\ \int_\RR \abs{\DIFF f^y}(G^y)\dd{y}\le\abs{\DIFF f}(G).
\end{equation}
	\subsection{Sobolev functions on and into metric measure spaces}
	
Let $(\XX,\dist,\mass)$ be a $\PI$ space and $(\YY,\rho)$ be a locally compact and separable metric space. Let $A\subseteq\XX$ open and $f:A\rightarrow \YY$. We say $f\in\W(A)$ (resp.\ $\Wloc(A)$) if $f\in\BV(A)$ (resp.\ $\BVloc(A)$) and moreover $\abs{\DIFF f}\ll\mass$. 

We justify now this approach, as defining Sobolev functions with integrability exponent $1$ on metric measure spaces is a rather delicate question (here we take advantage of the $\PI$ assumption). We will deal with only two cases: either $\XX=\RR^2$ or $\YY=\RR^2$, that is to say that either the domain or the codomain of the map will be Euclidean.
In the former case, our approach is proved to be equivalent to the one by Reshetnyak (\cite{Resh1997,Resh2004}) and then to the one by Korevaar and Schoen (\cite{KS93}). For the latter case, notice first that if $f=(f_1,f_2):A\rightarrow\RR^2$, then $f\in\W(A,\RR^2)$ if and only if $f_i\in\W(A,\RR)$ for $i=1,2$. Then, for a comprehensive discussion on equivalent definitions of the space $\W(A,\RR)$, we refer to \cite[Section 4.6]{DiM14a} and the references therein. In particular, it is possible to prove (see the proof of \cite[Theorem 4.6]{HKLL16} and \cite[Proposition 4.26]{Cheeger00}) that $f\in\W(A,\RR)$ if and only if we can find a sequence $\{f_n\}_n\subseteq\LIPloc(A)\cap\Lp^1(A)$ such that $f_n\rightarrow f$ in $\Lp^1(A)$ and $\lip(f_n)\rightharpoonup g$ weakly in $\Lp^1(A)$ for some $g\in\Lp^1(A)$. Also, it has been proved that $f\in\W(A,\RR)$ if and only if there exists $\tilde{f}:A\rightarrow\XX$ with $\tilde{f}=f$ $\mass$-a.e.\ and $\tilde{f}\in\mathrm{N}^{1,1}(A)$ (see \cite{Shanmugalingam00} for the definition of the Newtonian space $\mathrm{N}^{1,1}(A)$).

\subsection{Connectedness}\label{connsect}
Let $S$ be a topological space. We call connected components of $S$ the maximal connected subsets of $S$. If $S'\subseteq S$ and $a,b\in S\setminus S'$, we say that $S'$ separates $a$ and $b$ (in $S$) if $a$ and $b$ are in different connected components of $S\setminus S'$. 

We will be mainly interested in connectedness in $\RR^2$. Recall that an open set $A\subseteq\RR^2$ is connected if and only if it is connected by smooth curves. In particular, a closed set $C\subseteq\RR^2$ separates $a,b\in\RR^2\setminus C$ if and only if there exists no smooth curve $\gamma:[0,1]\rightarrow\RR^2\setminus C$ with $\gamma(0)=a$ and $\gamma(1)=b$. 

We will use the following result taken from \cite[V.14.3]{Newman}, that is a consequence of the so called Alexander lemma, see e.g.\ \cite[Theorem 30.1]{Moise23}, that we write below in a simplified version.
\begin{thm}\label{separates}
Let $C\subseteq\RR^2$ closed and $a,b\in\RR^2\setminus C$. If $C$ separates $a$ and $b$, then there exists a connected component of $C$ that separates $a$ and $b$.
\end{thm}
\begin{lemma}[Alexander]
Let $C_1,C_2\subseteq\RR^2$ closed and connected and let $a,b\in\RR^2\setminus (C_1\cup C_2)$. If $C_1\cup C_2$ separates $a$ and $b$, then either $C_1$ or $C_2$ separate $a$ and $b$.
\end{lemma}

	\section{Main results}
	\subsection{Jordan curves}
	\begin{prop}\label{Jordanmain}
		Let $(\XX,\dist,\mass)$ be a $2$-Ahlfors regular m.m.s.\ supporting a weak $(1,1)$-Poincaré inequality. Let $\gamma\subseteq \XX$ be a Jordan curve such that there exists an open set $D\subseteq\XX $ that is homeomorphic to an open set $E\subseteq\RR^2$ with $\gamma\subseteq D$. Assume moreover that $\gamma$ separates $D$ into exactly two connected components, say $A_1$ and $A_2$. Then
		$$\HH^1(\gamma)\le C\min\{\per(A_1,D),\ \per(A_2,D)\} $$
		where $C$ is a constant that depends only on the space $(\XX,\dist,\mass)$. 
		In particular, if $\per(A_1,D)<\infty$ or $\per(A_2,D)<\infty$, then $\mass(\gamma)=0$ so that $\per(A_1,D)=\per(A_2,D)$.
	\end{prop}
	
	\begin{proof}
		In the sequel we let $C$ denote a constant. It may vary during the proof, but in a way that depends only on the properties of the space $(\XX,\dist,\mass)$. 
		Of course, we can assume that $\per(A_1,D)<\infty$ or $\per(A_2,D)<\infty$, otherwise there is nothing to prove. For simplicity, set $A=A_1$ or $A=A_2$ so that $\per(A,D)<\infty$.
		
		For $g: D\rightarrow E$ the homeomorphism, from the statement we have that $g(\gamma)\subseteq\RR^2$ is a Jordan curve. Therefore, thanks to the Jordan–Schoenflies theorem (e.g.\ \cite{Cairns} or \cite[Theorem 10.4]{Moise23}), there is no loss of generality in assuming that $g(\gamma)=S^1\subseteq\RR^2$. We call $E_{\delta} = B^{\dist_e}_{\delta}(S^1)$ and $D_\delta = g^{-1}(E_{\delta})$. Then there exists a $\delta_0 >0$ such that $D_{5 \delta_0}\Subset D$ and the restrictions $g_{|D_{5 \delta_0}}$ and $g^{-1}_{|E_{5 \delta_0}}$ are uniformly continuous. In the following we replace $D$ with $D_{5 \delta_0} = g^{-1}(E_{5 \delta_0})$, $A$ with $A\cap D_{5 \delta_0} $ and $E$ with $E_{5 \delta_0}$. Clearly we still have $\gamma \subseteq D$ and $D$ is connected.
		
		We take now $\{u_n\}_n\subseteq\LIPloc(D)$ a sequence of discrete convolutions of the function $\chi_A\in\BV(D)\cap\Lp^1_{\mathrm{loc}}( D)$ (see \cite[Section 5]{KoLaLiSh17}). 
		Clearly, $u_n(x)\in[0,1]$ for every $x\in D$.  By the results on discrete convolutions, we know that \begin{equation}\label{tmp10}
		\int_D \lip(u_n)\dd{\mass}\le  C \per(A,D).
		\end{equation}
		Also, as a consequence of the compactness of $\gamma$ and the construction of discrete convolutions, there exists a sequence $\{r_n\}_n$ with $r_n\searrow  0$, $B_{r_n}(\gamma) \subseteq D$ and
		\begin{equation}\label{uguali}
		u_n(x)=\chi_A(x)\quad\text{if }x\notin B_{r_n}(\gamma).
		\end{equation}
		
		By \eqref{LipEsteq2} and \eqref{tmp10}, $$\int_0^1\HH^1( \{u_n=t\}\cap D)\dd{t}\le C\per(A,D)$$
		so that we can find $\{t_n\}_n\subseteq(0,1)$ such that if we define $$K_n\defeq\{u_n=t_n\}\cap D,$$ then 
		\begin{equation}\label{tmp9}
		\HH^1(K_n)\le C\per(A,D)\quad\text{for every } n.
		\end{equation}

Assume now that $n$ is big enough so that
		\begin{equation}\label{ins11}
		g(B^\dist_{r_n}(\gamma))\subseteq B^{\dist_e}_{\delta_0}(S^1),
		\end{equation}
		in particular, by \eqref{uguali},
				\begin{equation}\notag
			g(K_n)\subseteq B^{\dist_e}_{\delta_0}(S^1),
		\end{equation}
and so $K_n$ is compact.

		Take now any two points $\bar{P},\bar{Q}\in {E}\subseteq\RR^2$ with $$\dist_e(\bar{P},(0,0))<1- 2\delta_0\quad\text{and}\quad\dist_e(\bar{Q} , (0,0))>1+ 2\delta_0.$$ It is easy to show that $g(K_n)$ separates the points $\bar{P}$ and $\bar{Q}$ in $\RR^2$ (see Subsection \ref{connsect} for the definition). Indeed, by \eqref{uguali} and by the choice of $n$ in \eqref{ins11}, we have that $(u_n(g^{-1}(\bar{P})),u_n(g^{-1}(\bar{Q})))$ is either $(1,0)$ or $(0,1)$. Then, any smooth curve $\varphi:[0,1]\rightarrow\RR^2$ with $\varphi(0)=\bar{P}$ and $\varphi(1)=\bar{Q}$ must intersect $g(K_n)=g(\{u_n=t_n\}\cap D)$, as $t_n\in(0,1)$. 
		Then, using Theorem \ref{separates}, we have a connected component of $K_n$, call it $K_n'$, such that $g(K_n')$ separates $\bar{P}$ and $\bar{Q}$ and also, by \eqref{tmp9},
		\begin{equation}\label{togol}
		\HH^1(K'_n)\le C\per(A,D).
		\end{equation}
		
		We notice now that $K_n'\rightarrow \gamma$ in the Hausdorff sense (see \cite[Defintion 4.4.9]{AmbrosioTilli04}), this is to say that there exists a sequence $\{s_n\}_n$ with $s_n\searrow 0$ such that eventually
		\begin{equation}\notag
		K_n'\subseteq B^\dist_{s_n}(\gamma)\quad\text{and}\quad
		\gamma\subseteq B^\dist_{s_n}(K_n').
		\end{equation}
		Indeed, the fist inclusion (with $s_n\ge r_n$) follows from \eqref{uguali} whereas the second can be easily proved by contradiction using the uniform continuity of the maps $g$ and $g^{-1}$, the first inclusion and the fact that $g(K_n')$ separates $\bar{P}$ and $\bar{Q}$ for every $n$.
		We conclude now using Golab's semicontinuity theorem (\cite[Theorem 4.4.17]{AmbrosioTilli04}) to infer, using also \eqref{togol},
		\begin{equation*}
			\HH^1(\gamma)\le \liminf_n \HH^1(K_n')\le C \per(A,D).\qedhere
		\end{equation*}
	\end{proof}

If $(\XX,\dist,\mass)=(\RR^2,\dist_e,\LL^2)$ we can show with a simple cone-density argument (e.g.\ \cite[Remark 2 at page 214 and Theorem 15.19]{mattila_1995}) that the conclusion of Proposition \ref{Jordanmain} (in the case $\HH^1(\gamma)<\infty$) improves to the stronger form 
$$
	\HH^1(\gamma\Delta(\partial^*A_i\cap D))=0\quad\text{for }i=1,2.
$$
This stronger result, in the Euclidean framework, can be also obtained directly as a consequence of \cite[Theorem 7]{ambcasmas99}. Conversely, if we were able to obtain an estimate of this type in our context, then Proposition \ref{Jordanmain} would follow. We point out that in the proof above we did not pursue an optimal value for constant $C$: the outcome can be easily quantitatively improved.

	\subsection{Homeomorphisms of bounded variation}

	\begin{prop}\label{homeoweak}
		Let $(\XX,\dist,\mass)$ be a metric measure space such that there exists $C_H''>0$ with 
		\begin{equation}\label{masssmall}
		\mass(B_r(x))\le C_H'' r^2\quad\text{for every }x\in\XX\text{ and }r>0.
		\end{equation}
		Let moreover $G\subseteq\RR^2$ open, $\Omega\subseteq\XX$ open and $f:G\rightarrow\Omega$ a homeomorphism. If $f\in\BVloc(G)$, then $f^{-1}\in \BVloc(\Omega)$.
	\end{prop}
	\begin{proof}
		In the sequel we let $C$ denote a constant. It may vary during the proof, but in a way that depends only on the properties of the space $(\XX,\dist,\mass)$. 
		Notice that, as we have to prove that $f^{-1}\in\BVloc(\Omega)$, we can assume with no loss of generality that $G=B_1(0)$ and that $f\in\BV(G)$. We denote $f^{-1}=(g_1,g_2)$. 
		Now we follow the argument used in the proof of \cite[Theorem 1.3]{donoschi} to prove that $g_1\in\BVloc(\Omega)$. As the same argument applies to $g_2$, this will be enough to conclude the proof.
		
		Notice that all the sets $G_x$ for $x\in(-1,1)$ are segments, so that, by \eqref{slicingeq} and \eqref{diffeH} it follows that
		\begin{equation}\label{tmp11}
		\int_{-1}^1\HH^1(f_x(G_x))\dd{x}\le\abs{\DIFF f}(G)<\infty.
		\end{equation} Therefore, by Proposition \ref{coarea}, the conclusion will follow from
		\begin{equation}\label{perminhh}
		\per  (\{g_1>x\},\Omega)\le C \HH^1(f_x(G_x))\quad\text{for }\LL^1\text{-a.e.\ }x.
		\end{equation}
		
		We prove now \eqref{perminhh}. Let $\psi:\RR\rightarrow\RR$ be the Lipschitz function defined as
		\begin{equation}\notag
		\psi(x)=
		\begin{cases}
		1 \quad&\text{if\ }x\le1/3\\
		-3 x+ 2\quad&\text{if\ }x\in[1/3,2/3]\\
		0&\text{if\ }x\ge 2/3.
		\end{cases}
		\end{equation}
		Fix $x$ such that $L\defeq \HH^1(f_x(G_x))<\infty$ and (recall \eqref{leneH}) let $\gamma:[0,L]\rightarrow\Omega$ be the arc-length parametrization of the curve $f_x(\,\cdot\,)$. Notice that, if $n\in\NN, n>0$ and $r_n\defeq L/n$, we have that $$B_{r_n}(\gamma)\subseteq \bigcup_{t\in\{ 0,r_n,2r_n,\dots,n r_n\}} B_{3 {r_n}} (\gamma(t)).$$
		Using also \eqref{masssmall} it follows that $$\mass(B_{r_n}(\gamma))\le C (n+1) r_n^2 \le C L r_n.$$
		Notice that, being $f$ a homeomorphism, 
		\begin{equation}\label{toshort}
		\gamma=f_x(G_x)=f(\partial\{(u,v)\in G:u>x\})=\partial (f(\{ (u,v)\in G:u>x\}))=\partial \{g_1>x\}.
		\end{equation}  
		Now we can check \eqref{perminhh} using the sequence of locally Lipschitz function (depending on $n$)\begin{equation*}
\chi_{\{g_1>x\}}(\,\cdot\,)\vee\psi(r_n^{-1}\mathrm{dist}(\,\cdot\, , \partial\{g_1>x\})).\qedhere
		\end{equation*}
	\end{proof}

	\begin{proof}[Proof of Theorem~\ref{main}]
		Proposition \ref{homeoweak} shows that if $f\in\BVloc(G)$, then also $f^{-1}\in \BVloc(\Omega)$. We remark that, under the more restrictive hypothesis on the space in force here, we can shorten the proof of Proposition \ref{homeoweak} using \cite[Theorem 1.1]{lahti2020}, according to which \eqref{perminhh} immediately follows from \eqref{toshort}.
		
		We prove then the converse implication. In the sequel we let $C$ denote a constant. It may vary during the proof, but in a way that depends only on the properties of the space $(\XX,\dist,\mass)$. Notice that we can assume with no loss of generality that $G=(0,1)^2$ and $f^{-1}\in\BV(\Omega)$. 
		
		Assume for the moment that for every $\epsilon>0$ it holds 		\begin{equation}\label{concludingclaim}
		\int_\epsilon^{1-\epsilon}\HH^1(f_{x_0}((\epsilon,1-\epsilon)))\dd{x_0}<\infty
		\quad\text{and}\quad
		\int_\epsilon^{1-\epsilon}\HH^1(f^{y_0}((\epsilon,1-\epsilon)))\dd{y_0}<\infty.
		\end{equation} 
		We claim that this is enough to conclude.
		Take $\phi\in\LIP(\XX)$ $1$-Lipschitz. Notice that $\phi\circ f:\RR^2\rightarrow\RR$ and so it is well known that if $U\subseteq (\epsilon,1-\epsilon)^2$ is open we can estimate (with the obvious meaning for the restriction of $\phi\circ f$ to lines) $$\abs{\DIFF( \phi\circ f)}(U)\le \int_\epsilon^{1-\epsilon}\abs{ \DIFF( \phi\circ f)_x}(U_x)\dd{x}+\int_\epsilon^{1-\epsilon}\abs{ \DIFF( \phi\circ f)^y}(U^y)\dd{y}$$
		where $U_x\defeq\{y\in\RR:(x,y)\in U\}\subseteq (\epsilon,1-\epsilon)$ and similarly $U^y\defeq\{x\in\RR:(x,y)\in U\}\subseteq (\epsilon,1-\epsilon)$.
		Clearly,
		$$\abs{\DIFF (\phi\circ f)}(U)\le  \int_\epsilon^{1-\epsilon}\abs{ \DIFF f_x}(U_x)\dd{x}+\int_\epsilon^{1-\epsilon}\abs{ \DIFF f^y}(U^y)\dd{y}.$$
		It follows that the finite Borel measure defined as 
		$$ (\epsilon,1-\epsilon)^2\supseteq B\mapsto \int_\epsilon^{1-\epsilon}\abs{ \DIFF f_x}(B_x)\dd{x}+\int_\epsilon^{1-\epsilon}\abs{ \DIFF f^y}(B^y)\dd{y}$$
		is an upper bound for $\abs{\DIFF f}\mres (\epsilon,1-\epsilon)^2$, that reads as, taking into account \eqref{diffeH}, 
		\begin{equation}\label{maininmain}
		\abs{\DIFF f}\mres (\epsilon,1-\epsilon)^2 \le 	\int_\epsilon^{1-\epsilon}\HH^1(f_{x_0}((\epsilon,1-\epsilon)))\dd{x_0}+
		\int_\epsilon^{1-\epsilon}\HH^1(f^{y_0}((\epsilon,1-\epsilon)))\dd{y_0}.
		\end{equation}
	Therefore, being $\epsilon>0$ arbitrary, this shows that $f\in\BVloc(G)$.

		We prove now the first inequality in \eqref{concludingclaim} (as the second follows from the same argument). Denote $f^{-1}=(g_1,g_2)$.
		Fix now $x_1,y_0,y_1$ with $1-\epsilon<x_1<1$, $0<y_0<\epsilon$ and $1-\epsilon<y_1<1$ such that 
		$$\per(\{g_1<x_1\},\Omega),\ \per(\{g_2>y_0\},\Omega),\ \per(\{g_2<y_1\},\Omega)<\infty$$
		(this is possible thanks to \eqref{coareaeq}).
		
		Notice that, for any $x_0$ with $\epsilon<x_0<1-\epsilon$ such that $\per(\{g_1>x_0\},\Omega)<\infty$ (that is for $\LL^1$-a.e.\ $x_0$, thanks to \eqref{coareaeq}) we have that $\partial((x_0,x_1)\times(y_0,y_1))\subseteq\RR^2$ is a Jordan curve $\gamma_{x_0}$ such that the Jordan curve $f(\gamma_{x_0})$ encloses the set of finite perimeter $$\{x_0<g_1< x_1,\ y_0<g_2< y_1\}.$$
		By Proposition \ref{Jordanmain} (with $\Omega$ in place of $D$) and the submodularity of the perimeter we obtain \begin{equation}\label{tmp12}
		\begin{split}
		\HH^1(f(\gamma_{x_0}))\le C\Big(&\per(\{g_1>x_0\},\Omega)+\per(\{g_1<x_1\},\Omega)\\&+\per(\{g_1>y_0\},\Omega)+\per(\{g_2<y_1\},\Omega)\Big).
		\end{split}
		\end{equation} 
		Notice that \eqref{coareaeq} applied to $g_1$ yields that the right hand side of \eqref{tmp12} is integrable with respect to $\LL^1(\diff x_0)$ on $(\epsilon,1-\epsilon)$ and in light of this we can recall the trivial fact $$\HH^1(f(\gamma_{x_0}))\ge\HH^1(f_{x_0}((\epsilon,1-\epsilon)))$$ to prove our claim \eqref{concludingclaim}.
	\end{proof}
In the proofs above we sacrificed generality and the possibility to obtain a stronger result in favour of the simplicity of notation. However one may adapt the arguments to show the following
	\begin{thm}\label{mainimproved}
	Let $(\XX,\dist,\mass)$ be a $2$-Ahlfors regular metric measure space supporting a weak $(1,1)$-Poincaré inequality. Let moreover $G\subseteq\RR^2$ open, $\Omega\subseteq\XX$ open and $f:G\rightarrow\Omega$ a homeomorphism. Then there exists a constant $C$, depending only on the space $(\XX,\dist,\mass)$ (in particular, only on the constants appearing in \eqref{Ahlfors} and \eqref{Poincare}) such that
	\begin{equation}\label{maineq}
C^{-1}\abs{\DIFF f}(G) \le \abs{\DIFF f ^{-1}}(\Omega) \le C\abs{\DIFF f}(G)
	\end{equation}
\end{thm} 
\begin{proof}
The proof of this result is more or less a careful inspection of the arguments used above. We can clearly assume either $f\in\BV(G)$ or $f^{-1}\in\BV(\Omega)$ so that, thanks to Theorem \ref{main}, we know that both $\abs{\DIFF f}$ and $\abs{\DIFF f^{-1}}$ are (possibly infinite) Borel measures. 

In what follows $C$ will denote a constant that depends only on the properties of $(\XX,\dist,\mass)$ as described above, and, as usual, it may vary during the proof. First, we show that 
\begin{equation}\label{funda}
C^{-1}\abs{\DIFF f}(B) \le \abs{\DIFF f ^{-1}}(f(4 B)) \le C\abs{\DIFF f}(4 B)
\end{equation} 
whenever $B=B_r(x)\subseteq G\subseteq\RR^2$ is such that $4 B\subseteq G$ ($C$ is independent of such ball).
The last inequality is precisely the content of Proposition \ref{homeoweak}, indeed the centre and radius of the ball in consideration did not play any role in the proof of Proposition \ref{homeoweak}.
To prove the first inequality, for $r>0$ and $x\in\RR^2$ let $Q_r(x)\defeq B^{\dist_\infty}_r(x)$, i.e.\ the open cube in $\RR^2$ defined as follows $$Q_r((x_1,x_2))\defeq\left\{(y_1,y_2):\max\{\abs{x_1-y_1},\abs{x_2-y_2}\}<r\right\}.$$
Notice now that as $B_r(x)\subseteq Q_r(x)\subseteq Q_{2 r}(x)\subseteq B_{4 r}(x)$
we only have to show $$\abs{\DIFF f}(Q_r(x))\le C\abs{\DIFF f^{-1}} (f(Q_{2 r}(x))),$$
that is a consequence of \eqref{tmp12} and \eqref{coareaeq}. Namely, we only need to improve the argument used in the proof of Theorem \ref{main}. Using the notation of the proof of Theorem \ref{main} (specifically recall that $f^{-1} = (g_1,g_2)$), we can show along the same lines that if we consider, for every $x_0\in (-r,r)$, the curve $\gamma_{x_0}$, that is the boundary of the set $$\left\{x_0<g_1<3 r/2+x_0/2,-3 r /2 + x_0/2<g_2<3 r /2 + x_0/2\right\}$$ 
and we set, for simplicity, $\tilde{A}\defeq f(Q_{2 r}(x))$, then we have
\begin{equation}\notag
	\begin{split}
&\int_{-r}^r \HH^1(f_{x_0}(-r,r))\dd{x_0}\le \int_{-r}^r
C\Big(\per(\{g_1>x_0\},\tilde{A})+\per(\{g_1<{3r}/{2} + {x_0}/{2}\},\tilde{A})\\&\qquad+\per(\{g_2>-3r/2+x_0/2\},\tilde{A})+\per(\{g_2<3r/2+x_0/2\},\tilde{A})\Big)\dd{x_0}
	\end{split}
\end{equation}
and we can also prove a similar estimate for $\int_{-r}^r \HH^1(f^{y_0}(-r,r))$. Now the conclusion follows as in the proof of Theorem \ref{main}, employing a suitable variant of \eqref{maininmain}, that follows from the integral inequality above.

We show now how \eqref{funda} allows us to conclude.
Using Whitney-type covers for $G\subseteq\RR^2$ in the form described in \cite[Section 5]{KoLaLiSh17} (see more precise references therein) we obtain that there exists a sequence of balls $\{B_j\}_j$ such that $4 B_j\subseteq G$, and
\begin{equation}\label{Whit}
1\le\sum_j \chi_{B_j}\le\sum_j \chi_{4 B_j} \le C\quad \text{on }G.
\end{equation}
We sketch here a possible construction of such cover for the reader's convenience. Consider the family of balls $$\left\{B_{r_x}(x),\ r_x=\min\{1,\mathrm{dist}(x,\XX\setminus G)/25\}\right\}$$
and, using the Vitali covering lemma (see e.g.\ \cite[Theorem 1.2]{Heinonen01}) extract a sequence of pairwise disjoint balls $\{B_j=B_{r_j}(x_j)\}_j$ with $B_j\subseteq G \subseteq\cup_j 5 B_j$. We only have to show the bounded overlap property \begin{equation}\label{boundedoverlap}
\sum_j\chi_{20 B_j}\le C\quad\text{on }G,
\end{equation} then the claim will follow choosing $\{B_j'\defeq 5 B_j\}_j$.
Assume now $20 B_{\bar{i}}\cap 20 B_{\bar{j}}\ne\emptyset$. If $\mathrm{dist}(x_{\bar{i}},\XX\setminus U)\ge 25$ then $r_{\bar{i}}=1$. Otherwise, if $\mathrm{dist}(x_{\bar{i}},\XX\setminus U)< 25$, then
$$ 25 r_{\bar{j}}\le\mathrm{dist}(x_{\bar{j}},\XX\setminus U)\le \dist(x_{\bar{j}},x_{\bar{i}})+\mathrm{dist}(x_{\bar{i}},\XX\setminus U) \le 20 r_{\bar{i}}+20 r_{\bar{j}} +25 r_{\bar{i}} $$
so that 
$r_{\bar{i}} \ge C r_{\bar{j}}. $
To sum up, if $20 B_i\cap 20 B_j\ne \emptyset$ then $r_i\ge C r_j$. Now recall that the balls in $\{B_j\}_j$ are pairwise disjoint and that $\RR^2$ satisfies a doubling inequality, then a classical argument shows that \eqref{boundedoverlap} follows.

We conclude now using \eqref{Whit} twice, as by also \eqref{funda}, $$\abs{\DIFF f}(G)\le\sum_j\abs{\DIFF f}(B_j)\le \sum_j C\abs{\DIFF f^{-1}}(f(4 B_j))\le 4 C \abs{\DIFF f^{-1}}(f(G))$$
and similarly
\begin{equation*}
\abs{\DIFF f^{-1}}(f(G))\le \sum_j \abs{\DIFF f^{-1}}(f(B_j))\le C \sum_j \abs{\DIFF f} (B_j)\le C \abs{\DIFF f}(G).\qedhere
\end{equation*}
\end{proof}

\begin{proof}[Proof of Theorem \ref{mainplus}]
We notice that, up to shrinking $G$ and accordingly $\Omega$, we can just show that $f\in\W(G)$ if and only if $f^{-1}\in\W(\Omega)$. Assume then that  $f\in\W(G)$ or $f^{-1}\in\W(\Omega)$. Then, by Theorem \ref{mainimproved}, $\abs{\DIFF f}$ is a finite Borel measure on $G$ and $\abs{\DIFF f^{-1}}$ is a finite Borel measure on $\Omega$.

As the constant $C$ in \eqref{maineq} is independent of $G$ and $\Omega$ (once that $(\XX,\dist,\mass)$ is fixed) we have that for every $A\subseteq G$ open, \eqref{maineq} is satisfied with $A$ in place of $G$ and $f(A)$ in place of $\Omega$.
Therefore, taking into account \eqref{Luz} and \eqref{Luz1}, it follows $\abs{\DIFF f}\ll\LL^2$ if and only $\abs{\DIFF f^{-1}}\ll\mass$, that is our claim.
\end{proof}

	We have proved our main Theorem \ref{main} under the assumption that $(\XX,\dist,\mass)$ is a $2$-Ahlfors regular m.m.s.\ supporting a weak $(1,1)$-Poincaré inequality. One may wonder if $2$-Ahlfors regularity is really needed or if it is enough a doubling inequality. Recall that in our proof we needed $2$-Ahlfors regularity to use the fact that $\HH^h$ and $\HH^1$ are comparable. We give below two elementary examples in which we show that $2$-Ahlfors regularity is indeed necessary. The first example will show necessity of the upper bound $\mass(B_r(x))\le C_H'' r^2$ whereas the second deals with necessity of the lower bound $\mass(B_r(x))\ge C_H' r^2$. 
	We will use the fact (see e.g.\ \cite[Appendix A.2, pag.\ 347]{Bjorn-Bjorn11}) that if $\nu>-n$, then $(\RR^n,\dist_e,\abs{\,\cdot\,}^{\nu}\LL^n)$ is a doubling space supporting a weak $(1,1)$-Poincaré inequality, where $\abs{\,\cdot\,}$ denotes the Euclidean norm. In particular, $(\RR^2,\dist_e,\abs{\,\cdot\,}^{-3/2}\LL^n)$ and $(\RR^2,\dist_e,\abs{\,\cdot\,}\LL^n)$ belong to this class of spaces. Will will identify $\RR^2\simeq\mathbb{C}$ and use coordinates $(x_1,x_2)$ as well as polar coordinates $(r,\theta)$.
	\begin{ex*}
		Let $(\XX,\dist,\mass)\defeq(\RR^2,\dist_e,\abs{\,\cdot\,}^{-3/2}\LL^2)$. Notice that $(\XX,\dist,\mass)$ does not satisfy the upper $2$-Ahlfors bound.
		Let now $G\defeq B_1(0)\subseteq\RR^2$ and $\Omega\defeq B_1(0)\subseteq\XX$.
		Define an homeomorphism $$f:G\rightarrow \Omega\quad f(r e^{i\theta})\defeq r^2 e^{i\theta }$$
		whose inverse is $$f^{-1}:\Omega\rightarrow G\quad f^{-1}(r e^{i\theta})\defeq \sqrt{r} e^{i\theta }.$$
		We may compute,
		$$
			\frac{\partial f_1}{\partial x_1}=2 r \cos^2(\theta)+r\sin^2(\theta)
		$$
		and similarly compute the other entries of the Jacobian matrix $\DIFF f$.
		Then it is clear that $$\abs{\DIFF f}\in\Lp^\infty (B_1(0);\LL^2)$$
		so that $f\in\BV(G;\LL^2)$
		while a simple computation yields
		$$
		\frac{\partial (f^{-1})_2}{\partial x_1}=-\frac{1}{2 \sqrt{r}}\sin(\theta)\cos(\theta)
		$$ 
		so that
		$$\abs{\DIFF f^{-1}}\notin\Lp^1 (B_\eta(0);\abs{\,\cdot\,}^{-3/2}\LL^2)\quad\text{for every }\eta>0$$
		and it follows $f^{-1}\notin\BVloc(\Omega;\mass)$.
	\end{ex*}
	\begin{ex*}
		Let $(\XX,\dist,\mass)\defeq(\RR^2,\dist_e,\abs{\,\cdot\,}\LL^2)$. Notice that $(\XX,\dist,\mass)$ does not satisfy the lower $2$-Ahlfors bound. 	Let now $G\defeq B_1(0)\subseteq\RR^2$ and $\Omega\defeq B_1(0)\subseteq\XX$. Define an homeomorphism $$f:G\rightarrow \Omega\quad f(r e^{i\theta})\defeq r e^{i(\theta+1/r^2) }$$
		whose inverse is $$f^{-1}:\Omega\rightarrow G\quad f^{-1}(r e^{i\theta})\defeq  r e^{i(\theta-1/r^2) }.$$
		We may compute
		$$
			\frac{\partial f_1}{\partial x_1}=\cos(\theta)\cos(\theta+1/r^2)+\sin(\theta)\sin(\theta+1/r^2)+\frac{2}{r^2}\cos(\theta)\sin(\theta+1/r^2).
		$$
		Then it is clear that $$\abs{\DIFF f}\notin\Lp^1 (B_\eta(0);\LL^2)\quad \text{for every }\eta>0$$ as, if $\eta>0$, $$\frac{1}{r}\sin(\theta+1/r^2)\notin \Lp^1((0,\eta);\LL^1)\quad\text{for }\LL^1\text{-a.e.\ }\theta\in(0,2\pi),$$
		a fast way to prove this is to bound $\frac{1}{r}{\sin^2(\theta+1/r^2)}\le \frac{1}{r}\abs{\sin(\theta+1/r^2)}$ and then use the basics trigonometric inequality.
		Therefore $f\notin\BVloc(G;\LL^2)$. 
		Similarly to what done above,
		$$
			\frac{\partial (f^{-1})_1}{ \partial x_1} = \cos(\theta) \cos(\theta-1/r^2) + \sin(\theta)\sin(\theta-1/r^2) - \frac{2}{r^2}\cos(\theta)\sin(\theta-1/r^2)
		$$
		and, as the other entries of the Jacobian matrix $\DIFF f^{-1}$ have a similar same form,
		$$\abs{\DIFF f^{-1}}\in\Lp^1 (B_1(0);\abs{\,\cdot\,}\LL^2)$$
		and it follows $f^{-1}\in\BV(\Omega;\mass)$.
		\end {ex*}
		
		\subsection*{Acknowledgements}
		The authors would like to thank Luigi Ambrosio for his mentorship and help. Also they would like to thank him for suggesting the collaboration and putting the authors in contact.
	\bibliographystyle{alpha}
	\bibliography{Biblio1}

\end{document}